\documentclass{amsart}
\usepackage{amssymb}
\usepackage{amsthm}
\usepackage{graphicx}
\usepackage{enumitem}
\usepackage[maxnames=5, minnames=5]{biblatex}

\theoremstyle{definition}
\newtheorem{conj}{Conjecture}
\newtheorem{prop}{Proposition}
\newtheorem{lem}{Lemma}

\newtheorem{rem}{Remark}
\newtheorem{thm}{Theorem}

\newtheorem{definition}{Definition}

\DeclareMathOperator{\diam}{diam}

\DeclareMathOperator{\Vol}{Vol}

\DeclareMathOperator{\UW}{UW}

\DeclareMathOperator{\Sc}{Sc}
\DeclareMathOperator{\FillRad}{FillRad}

\title{Macroscopic Scalar Curvature in High Dimensions}
\author{Brendan Isley}
\date{}

\bibliography{bibliography}

\begin{document}

\maketitle
\begin{abstract}
    We describe a version of positive macroscopic scalar curvature motivated by the work of Alpert, Balitskiy, and Guth, and prove that this condition on a manifold implies a bound on its $1$-width in terms of its first Betti number. A key tool in the proof is a decomposition of any closed manifold into a family of chains with convenient combinatorial structure, which was inspired by Nabutovsky, Rotman, and Sabourau's work on sweepouts. Finally, using techniques developed by Chodosh, Li, and Liokumovich, we show that for a sufficiently connected manifold, if the universal cover satisfies this curvature condition, then the manifold has a finite cover homotopy equivalent to $S^n$ or connected sums of $S^{n-1} \times S^1$.
\end{abstract}

\section{Introduction}

\subsection{Background}
Given a Riemannian manifold $M^n$, the \emph{scalar curvature} at a point $x \in M$ is the unique real number $\Sc_M(x)$ which satisfies the following formula:
\begin{equation}\label{sc}
    \Vol(B(x,r)) = \omega_n r^n \left( 1 - \frac{\Sc_M(x)}{6(n+2)}r^2 + O(r^3)\right),
\end{equation}
where $B(x,r)$ is the ball of radius $r$ centered at $x$, and $\omega_n$ is the volume of a unit ball in $\mathbb{R}^n$. Equivalently, $\Sc_M(x)$ equals the trace of the Ricci tensor at $x$.

It is conjectured that manifolds with positive scalar curvature have a particular type of macroscopic geometry. The natural tool, in this context, to describe macroscopic geometry is through the notions of Urysohn width and macroscopic dimension. Given a metric space $X$ and a simplicial complex $Y$, the width of a map $f: X \to Y$ is the supremal diameter of its fibers, i.e. $\sup_{y\in Y}\diam_X(f^{-1}(y))$.
The \emph{$q$-dimensional Urysohn Width} (also referred to as the \emph{$q$-width}), denoted $\UW_q(X)$, is the infimal width over maps from $X$ into any $q$-dimensional complex. Finally, the \emph{macroscopic dimension} of $X$, denoted $\dim_{mc}(X)$, is the minimal $q$ such that $\UW_q(X)$ is finite. Macroscopic dimension roughly describes the dimension of a space ``from a distance." For instance, the macroscopic dimension of the sphere $S^2$ is zero, of the cylinder $S^1 \times \mathbb{R}$ is one, and of the plane $\mathbb{R}^2$ is two.

The hypothesized connection between positive scalar curvature and macroscopic geometry is described by an unresolved conjecture of Gromov.
\begin{conj}[Gromov 1988, \cite{gromov1988width, gromov1996positive}]\label{psc-conj}
    There exists a dimensional constant $c_n$ such that the following holds. Suppose $M$ is an $n$-dimensional manifold, and suppose there exists a positive constant $s$ satisfying $\Sc_M(x) \geq s^2 > 0$ for all $x \in M$. Then the $(n-2)$-dimensional Urysohn width of $M$ satisfies
    $$
    \UW_{n-2}(M) \leq \frac{c_n}{s}.
    $$
    In particular, the macroscopic dimension of $M$ is at most $n-2$.
\end{conj}

The $2$-dimensional case of this conjecture is a direct consequence of the Bonnet-Myers theorem, and the $3$-dimensional case has also been proven in \cite{gromov1983positive,liokumovich2023waist,chodosh2024generalized}. All other dimensions are unresolved. Chodosh, Li, and Liokumovich provided partial progress towards higher dimensional cases in \cite{chodosh2023classifying}, where they show that a class of $4$ and $5$-dimensional manifolds with positive scalar curvature and restricted topology have macroscopic dimension bounded above by $1$.

Equation (\ref{sc}) says that for spaces with positive scalar curvature, the volumes of \emph{infinitesimal} balls are less than the volume of Euclidean balls. We could  instead require a bound on the volume of balls of a \emph{fixed} radius, which is a macroscopic curvature condition. We want to know whether the macroscopic geometry of manifolds with ``positive macroscopic scalar curvature'' is similar to that of manifolds with positive scalar curvature.

Gromov conjectured the following result in \cite{gromov2006large}, which was proven by Larry Guth in \cite{guth2017volumes}. Later proofs were given in \cite{papasoglu2020uryson} and \cite{liokumovich2022filling}.
\begin{thm}[Guth 2017, \cite{guth2017volumes}]
    There exists a dimensional constant $c(n)$ such that the following holds. If $M$ is a closed $n$-dimensional manifold and $r>0$ is a constant such that all balls of radius $r$ in $M$ have volume at most $c(n)\,r^n$, then $\UW_{n-1}(M) \leq r$.
\end{thm}

One may conjecture that the same result should be true if $(n-1)$-width is replaced by $(n-2)$-width, due to the use of $(n-2)$-width in Conjecture \ref{psc-conj}. However, this is false: embed an $(n-1)$-dimensional disk of large radius in $\mathbb{R}^{n+1}$, and take the boundary of a small neighbourhood of the disk. This boundary would be an $n$-dimensional manifold with  $(n-1)$-width on the order of the radius of the disk, but would have unit balls  with small volume (see \cite{alpert2024macroscopic}).

In this example, the topologies of the balls we measure are nontrivial, which is different from the situation of measuring infinitesimal balls in the definition of scalar curvature, which must have trivial topology. This suggests that a stronger definition of macroscopic scalar curvature which restricts the topology of balls could potentially bound the $(n-2)$-width. Alpert, Balitskiy, and Guth proposed alternate ways to measure macroscopic scalar curvature in \cite{alpert2024macroscopic}. One of their suggestions was to measure the volume of balls in the universal cover of larger balls, but their conjecture using this method was recently disproven by Kumar and Sen in \cite{kumar2026urysohn}. Another suggestion of theirs was to define positive macroscopic scalar curvature by bounding the volume of balls, as well as imposing a condition on loops which are contained within small balls. 

Let us describe the latter condition more precisely. Let $M$ be a Riemannian manifold, and let $G=\mathbb{Z}$ if $M$ is orientable and $G=\mathbb{Z}_2$ if $M$ is nonorientable. By a \emph{volume condition} (VC) on $M$, we mean a choice of positive constants $r_0$ and $v_0$ which satisfy the following property:
\begin{itemize}
    \item[(VC)] Balls of radius $r_0$ have volume less than $v_0$.
\end{itemize}
By a \emph{filling condition} (FC) on $M$, we mean a choice of positive constants $r_1$ and $t_1$ which satisfy the following property:
\begin{itemize}
    \item[(FC)] Any $1$-cycle which is contained in a ball of radius $r_1$ is null-homologous with $G$-coefficients in the concentric ball of radius $t_1$.
\end{itemize}

The suggestion of Alpert, Balitskiy, and Guth was that these two conditions together form the correct definition of positive macroscopic scalar curvature. They conjectured the following.

\begin{conj}[Alpert--Balitskiy--Guth 2024, \cite{alpert2024macroscopic}]\label{pmsc-conj}
    For any dimension $n$ and any parameters $0<r_1<t_1$, there exist parameters $r_0,v_0$ and a constant $c(n,r_1,t_1)$ such that the following holds.
    
    Suppose $M$ is a complete $n$-dimensional manifold which satisfies (VC) with parameters $r_0,v_0$ and (FC) with parameters $r_1,t_1$.
    Then 
    $$\UW_{n-2}(M) \leq c(n,r_1,t_1).$$
    In particular, the macroscopic dimension of $M$ is at most $n-2$.
\end{conj}

They proved a weaker version of this conjecture when $n=3$, where instead of a constant $c(n,r_1,t_1)$ which only depends on the filling condition (FC) and the dimension $n$, they have a constant which also depends on the topology of $M$. 

Throughout this paper, unless otherwise stated, we will use homology coefficients $G = \mathbb{Z}$ if $M$ is orientable and $G=\mathbb{Z}_2$ if $M$ is nonorientable.

\begin{thm}[Alpert--Balitskiy--Guth 2024, \cite{alpert2024macroscopic}]\label{ABG-thm}
  For any parameters $0<r_1<t_1$, there exist parameters $r_0,v_0$ such that the following holds.
    
    Suppose $M$ is a complete $3$-dimensional manifold which satisfies (VC) with parameters $r_0,v_0$ and (FC) with parameters $r_1,t_1$.
    Then
    $$\UW_{1}(M) \leq 12(b+1)t_1,$$
    where $b = \dim H_1(M;G)$. In particular, if $b$ is finite, then $M$ has macroscopic dimension at most $1$.
\end{thm}

\subsection{Main Result}

The main theorem of this paper generalizes the $1$-width bound from the previous theorem to higher dimensions, by imposing a stronger condition than that of the previous theorem.

Suppose we are dealing with $n$-dimensional manifolds. We will use the same volume condition (VC), however by \emph{filling conditions} (FC'), we mean a choice of pairs of constants $(r_k,t_k)_{k=1}^{n-2}$ (with $r_k > 2t_{k-1}$) satisfying the following property:
\begin{itemize}
    \item[(FC')] Any $k$-cycle which is contained in a ball of radius $r_k$ is null-homologous with $G$-coefficients in the concentric ball of radius $t_k$.
\end{itemize}

Of course, the two conditions (VC) and (FC') are stronger than the two conditions (VC) and (FC) in Conjecture \ref{pmsc-conj} and Theorem \ref{ABG-thm}. Our main theorem shows that this stronger condition is enough to bound the $1$-width of manifolds of arbitrary dimension, assuming a bound on the first Betti number. 
\begin{thm}[Main Theorem]\label{main-thm}
    For any dimension $n$ and any parameters $0 < r_k < t_k$ (for all $k \in \{1,\dots,n-2\}$) satisfying $r_k > 2t_{k-1}$, there exist parameters $r_0,v_0$ such that the following holds.
    
    Suppose $M$ is a complete $n$-dimensional manifold which satisfies (VC) with parameters $r_0,v_0$ and (FC') with parameters $(r_k,t_k)_{k=1}^{n-2}$.
    Then
    $$\UW_{1}(M) \leq 24(b+1)t_{n-2},$$
    where  $b = \dim H_1(M;G)$. In particular, if $b$ is finite, then the macroscopic dimension of $M$ is at most $1$.
\end{thm}

Our proof will show that we can take $r_0 = 10t_{n-2}$ and $v_0$ satisfying
\begin{equation}\label{dependency}
\frac{v_0}{r_0^2} < \frac{1}{50} \left(\frac{\underset{1 \leq i \leq n-2}{\min} \{r_i-2t_{i-1} \}} {2^{n-2}(n-1)(n-2)}\right)^{n-2}, 
\end{equation}
where we take $t_0 = 0$.

If we have the same macroscopic curvature condition in the universal cover of a sufficiently connected manifold, we can conclude strong conditions about the topology of the manifold.

\begin{thm}\label{topcor}
    For any dimension $n$ and any parameters $0 < r_k < t_k$ (for all $k \in \{1,\dots,n-2\}$) satisfying $r_k > 2t_{k-1}$, there exist parameters $r_0,v_0$ such that the following holds.
    
    Suppose $M$ is a closed $n$-dimensional manifold whose universal cover $\widetilde{M}$ satisfies (VC) with parameters $r_0,v_0$ and (FC') with parameters $(r_k,t_k)_{k=1}^{n-2}$. Suppose further that $\pi_2(M) = \dots = \pi_{n-2}(M) = 0$.
    Then some finite cover of $M$ is homotopy equivalent to either $S^n$ or connected sums of $S^{n-1} \times S^1$.
\end{thm}

Similar to Theorem \ref{main-thm}, we can take any $r_0,v_0$ satisfying inequality (\ref{dependency}).

The proof of Theorem \ref{main-thm} has three main steps, which will make up the next three sections of this paper. In Section \ref{ABG-sec}, we generalize the Alpert-Balitskiy-Guth argument to higher dimensions, and show it suffices to fill any $(n-2)$-dimensional submanifold $S$ which has small volume. In Section \ref{fillrad-sec}, we adapt a construction of Nabutovsky--Rotman--Sabourau  to decompose the submanifold $S$ in a convenient way. In Section \ref{proof-sec}, we use this decomposition to fill in $S$ by breaking it into small pieces that can be filled, hence proving Theorem \ref{main-thm}. Finally, in Section \ref{top-sec}, we prove Theorem 4 by modifying an argument of Chodosh, Li, and Liokumovich (the proofs of \cite[Theorem 1 and Corollary 14]{chodosh2023classifying}, which uses results from \cite{gadgil2009topology}). We also use a connection between macroscopic dimension and virtually free fundamental groups (Theorem \ref{virtfree-thm}) which seems to have been known by experts.

\subsection*{Acknowledgements} The author would like to thank his advisor, Yevgeny Liokumovich, for unending patience and support with this project. The author would also like to thank Talant Talipov, Amal Vayalinkal, and Kevin Santos for numerous helpful conversations.

\section{Bounding Width by Filling Submanifolds}\label{ABG-sec}

In this section, we generalize the argument of Theorem \ref{ABG-thm} provided by Alpert, Balitskiy, and Guth \cite{alpert2024macroscopic} to $n$-dimensional manifolds.

\begin{prop}\label{ABG-prop}
    For any parameters $v_1,t_1,$ there exist parameters $r_0 = 5t_1$ and $v_0 = \frac{v_1\cdot t_1^2}{2}$ such that the following holds.
    
    Suppose $M$ is a complete $n$-dimensional manifold which satisfies (VC) with parameters $r_0,v_0$, as well as the following filling condition:
    \begin{itemize}
        \item[(FC'')] Any closed $(n-2)$-submanifold (orientable if $M$ is) with volume less than $v_1$ is null-homologous with $G$-coefficients within a neighbourhood of radius $t_1$, where $G = \mathbb{Z}$ if $M$ is orientable and $G=\mathbb{Z}_2$ otherwise.
    \end{itemize}
    Then
    $$
    \UW_1(M) \leq 12(b+1)t_1,
    $$
    where $b = \dim H_1(M;G)$.
\end{prop}

We denote by $B_r(A)$ the tubular neighbourhood of a set $A$ with radius $r$. Moreover, $S_r(A)$ denotes $\partial B_r(A)$, consisting of all points with distance $r$ from $A$.

\begin{proof}
Suppose $M$ satisfies 
\begin{itemize}
    \item[(VC)] Balls of radius $5t_1$ have volume less than $\frac{v_1 \cdot t_1^2}{2}$
\end{itemize}
as well as (FC''). We will prove that, for any point $p\in M$, each connected component of each level set of the function $d(\cdot, p)$ has diameter at most $12(b+1)t_1$. Then the Reeb graph $Y$ is constructed by quotienting each connected component of each level set, and the quotient map $\pi: M \to Y$ would have fibers of diameter at most $12(b+1)t_1$, proving the desired $1$-width bound. 

Fix $p \in M$ and consider the metric spheres $S_t(p)$, $t > 0$. Suppose to the contrary that one of them has a connected component $C$ of diameter greater than $12(b + 1)t_1$. Necessarily this requires $t > 6t_1$.  We can assume without loss of generality that this component $C$ is path-connected, as otherwise we can pick a path in a small open neighbourhood, and all the inequalities below will still hold up to an arbitrarily small error. Hence, we assume there is a path $\gamma$ in $C$ connecting two points that are greater than distance $12(b+1)t_1$ apart.

Pick points $x_0, x_1, \ldots, x_{b + 1}$ along $\gamma$ such that the distance in $M$ between every two of them is greater than $12t_1$. For each $i = 0, \dots, b + 1$, let $\eta_i$ be the minimizing geodesic of length $t$, connecting $x_i$ to $p$.  
We have $b + 1$ triangles, each consisting of consecutive geodesics $\eta_{i-1}$ and $\eta_i$ together with the portion of $\gamma$ from $x_{i-1}$ to $x_i$. Their classes in $H_1(M;G)$ must be linearly dependent, so some nontrivial linear combination of them must be zero in homology. We let $z$ be such a homologically trivial cycle. There is some $i^*$ such that the net contribution of $\eta_{i^*}$ to $z$ has nonzero coefficient. We want to construct an $(n-1)$-cycle that intersects $z$ transversely once along $\eta_{i^*}$ and nowhere else.

Without loss of generality, we can assume the functions $d(x_{i^*}, \cdot)$ and $d(\eta_{i^*}, \cdot)$ are smooth (by approximating them by smooth functions, with an arbitrarily close approximation). We claim that there is a choice $(\rho_1, \rho_2) \in (0, t_1)\times (0, t_1/2)$ for which the following properties hold.
\begin{enumerate}
\item $S_{4t_1 + \rho_1}(x_{i^*})$ is an embedded hypersurface,
\item $S_{4t_1 + \rho_1}(x_{i^*}) \cap S_{t_1+\rho_2}(\eta_{i^*})$ is an embedded union of disjoint $(n-2)$-dimensional submanifolds,
\item These submanifolds have total $(n-2)$-volume less than $v_1$.
\end{enumerate}
For the first two properties, by Sard's theorem we know that the regular values of the (assumed smooth) functions $d(x_{i^*}, \cdot)$ and $d(x_{i^*}, \cdot) \times d(\eta_{i^*}, \cdot)$ have full measure.  For the third property, we use the coarea inequality on the function $f(q) = (d(x_{i^*}, q) - 4t_1, d(\eta_{i^*}, q) -t_1)$, which is $1$-Lipschitz in each coordinate:
\begin{align*} \int\limits_{(s_1, s_2) \in (0, t_1)\times(0,t_1/2)} \Vol_{n-2}(f^{-1}(s_1, s_2))\ ds_1ds_2 \leq& \Vol_n (f^{-1}((0, t_1)\times(0,t_1/2)))\\
 \leq& \Vol_n (B_{5t_1}(x_{i^*})) < \frac{v_1\cdot t_1^2}{2},
 \end{align*}
where the last inequality uses the volume condition (VC). We can choose a pair $(\rho_1, \rho_2)\in (0, t_1)\times(0,t_1/2)$ such that the corresponding union of $(n-2)$-dimensional submanifolds $f^{-1}(\rho_1, \rho_2)$ has at most average total volume, and thus has total volume less than $\frac{v_1 \cdot t_1^2}{2}\cdot \frac{1}{t_1 \cdot (t_1/2)} = v_1$, while simultaneously guaranteeing that $\rho_1$ and $(\rho_1, \rho_2)$ are regular values. Thus, the three desired properties hold.
 
 Let $\Sigma$ be the connected component of $S_{4t_1 + \rho_1}(x_{i^*}) \cap B_{t_1 + \rho_2}(\eta_{i^*})$ that intersects $\eta_{i^*}$. That intersection is orthogonal and is at a single point.  Let $\Gamma$ be the union of those components of $S_{4t_1 + \rho_1}(x_{i^*}) \cap S_{t_1 + \rho_2}(\eta_{i^*})$ that constitute the boundary of $\Sigma$. Because $\Gamma$ has volume less than $v_1$, the filling condition (FC'') guarantees that $\Gamma$ can be filled in a neighbourhood of radius $t_1$ by an $(n-1)$-chain $\Sigma'$. We will be done if we show that $\Gamma$ stays at least distance $t_1$ from $z$, because then gluing $\Sigma$ and $\Sigma'$ along $\Gamma$ produces an $(n-1)$-cycle with $G$-coefficients whose algebraic intersection number with $z$ is nonzero. This would contradict that $z$ is null-homologous.
 
 First we check that $\Gamma$ stays at least distance $t_1$ away from $\gamma$.  Let $a$ be a point on $\Gamma$.  There is a point $c$ on $\eta_{i^*}$ with $d(c, a) \leq t_1 + \rho_2$.  By the reverse triangle inequality on $x_{i^*}$, $a$, and $c$, we have $d(x_{i^*}, c) \geq 3t_1 + \rho_1 - \rho_2$.  Because $\eta_{i^*}$ is a geodesic, we have $d(x_{i^*}, c) + d(c, p) = t$. Hence, the triangle inequality on $c$, $p$, and $a$ gives $d(p, a) \leq t - 2t_1 + 2\rho_2 - \rho_1\leq t - t_1$, and so $a$ must be at least distance $t_1$ from $S_t(p)$ and thus from $\gamma$.
 
 Next we check that $\Gamma$ stays at least distance $t_1$ away from each $\eta_j$ with $j \neq i^*$.  Let $a$ be a point on $\Gamma$, and suppose that there is a point $c$ on $\eta_j$ with $d(a,c) \leq t_1$.  Then the triangle inequality on $x_{i^*}$, $a$, and $c$ gives $d(x_{i^*}, c) \leq 5t_1 + \rho_1$, so the reverse triangle inequality on $x_{i^*}$, $c$, and $p$ gives $d(p, c) \geq t - (5t_1 + \rho_1)$.  Because $d(p, c) + d(c, x_j) = t$, we have $d(c, x_j) \leq 5t_1 + \rho_1$.  Thus, the triangle inequality on $x_{i^*}$, $c$, and $x_j$ gives $d(x_{i^*}, x_j) \leq 10t_1 + 2\rho_1 < 12t_1$, contradicting $d(x_{i^*}, x_j) > 12t_1$.
 
 Thus, filling $\Gamma$ in its $t_1$-neighbourhood and adding this filling to $\Sigma$ gives an $(n-1)$-cycle with $G$-coefficients whose algebraic intersection number with $z$ is nonzero, contradicting the fact that $z$ is null-homologous.
\end{proof}

\section{Decomposing Submanifolds}\label{fillrad-sec}
We now shift our attention to filling in $(n-2)$-submanifolds of $M$ with small volume, within a small neighbourhood. Let us denote a submanifold of $M$ by $S^d$, for $d = n-2$. The main result of this section is Lemma \ref{decomp-lem}. This decomposition lemma allows us to break $S$ up into small chains with convenient combinatorial structure, and in Section \ref{proof-sec} this will be our key tool to fill $S$ within a small neighbourhood and hence prove Theorem \ref{main-thm}. For the remainder of this section, however, we treat $S$ as its own independent manifold, as opposed to a submanifold of $M$. In other words, the remainder of this section applies when $S$ is an arbitrary manifold, of arbitrary dimension $d$. The chains will have coefficients $G = \mathbb{Z}$ if $S$ is orientable, and $G = \mathbb{Z}_2$ if $S$ is nonorientable. If $S$ is orientable then $\mathbb{Z}_2$ coefficients also work; this may be applied in Section 4 if $M$ is nonorientable but $S$ happens to be orientable.

Let us describe the main idea behind the decomposition lemma. Suppose $S$ lies in a higher dimensional ambient manifold $P$. One strategy for breaking $S$ up into small pieces is to triangulate $P$ finely, then take the small pieces of $S$ to be the parts of $S$ which lie within the small simplices of $P$. The issue is that there may not exist an ambient manifold $P$ that breaks up $S$ into small pieces which can be filled easily, which is what we are trying to do with this decomposition. So what one could hope to do instead is find a different manifold $N$ which is broken into nicer pieces by the triangulation of $P$, then ``project" these pieces (using a degree-one map) onto $S$. This is the strategy we take, but instead of manifolds we need to widen our scope to pseudomanifolds, and instead of triangulations we work with cubical structures. It should be possible to use triangulations instead, possibly improving the constant in our main theorem, but we choose to use a cubical structure as it permits easier computations in coordinates. 

The machinery developed in this section heavily uses material which was developed by Nabutovsky, Rotman, and Sabourau \cite[Sections 2 and 3]{nabutovsky2021sweepouts}.
\subsection{Construction in the Cube}\label{cube-subsec}

Let $C^{k}=[-1,1]^{k}$ be the standard cubical $k$-simplex, also called the $k$-cube. Define the space
$$
X^k = \{ x \in C^k : \text{there exist } i_1,i_{2} \in \{1,\dots,k\} \text{ distinct such that } |x_{i_1}|,|x_{i_2}| \leq \frac{1}{2} \}
$$
formed of the points of $C^k$ with at least two coordinates bounded above by $\frac{1}{2}$ in absolute value. This space is a tubular neighbourhood of the skeleton dual to the $1$-skeleton of $C^k$. Next, define the space
$$
Z^k = \{x \in C^k : \text{there exist } i_1, \dots, i_{k-1} \in \{1,\dots,k \} \text{ distinct such that }|x_{i_p}| \geq \frac{1}{2} \text{ for all }p\}
$$
formed of the points of $C^{k}$ with at least $k-1$ coordinates bounded below by $\frac{1}{2}$ in absolute value. This space is a tubular neighbourhood of the $1$-skeleton $(C^k)^{(1)}$ of $C^k$. Finally, let $Y^{k-1}$ be the intersection of $X^k$ and $Z^k$. This space is the cubical $(k-1)$-dimensional complex which is the boundary of $X^k$ minus the relative interior of $X^k \cap \partial C^k$. It can also be expressed as the union
\begin{align*}
{Y}^{k-1} =  \cup_{p=2}^k  \{ x \in C^{k} : \mbox{there exist } i_1,\cdots,i_p \in \{1,\dots,k\} & \mbox{ distinct such that } \nonumber \\
 &  |x_{i_1}| \leq \frac{1}{2},  \nonumber \\
 & |x_{i_2}| = \cdots = |x_{i_p}| = \frac{1}{2}, \nonumber \\
 & |x_i| > \frac{1}{2} \mbox{ for every } i \neq i_1,\cdots, i_p \}
\end{align*}
of cubical $(k + 1 - p)$-complexes with $2 \leq p \leq k$.
\begin{figure}[h]
    \centering
    \includegraphics[width=0.5\linewidth]{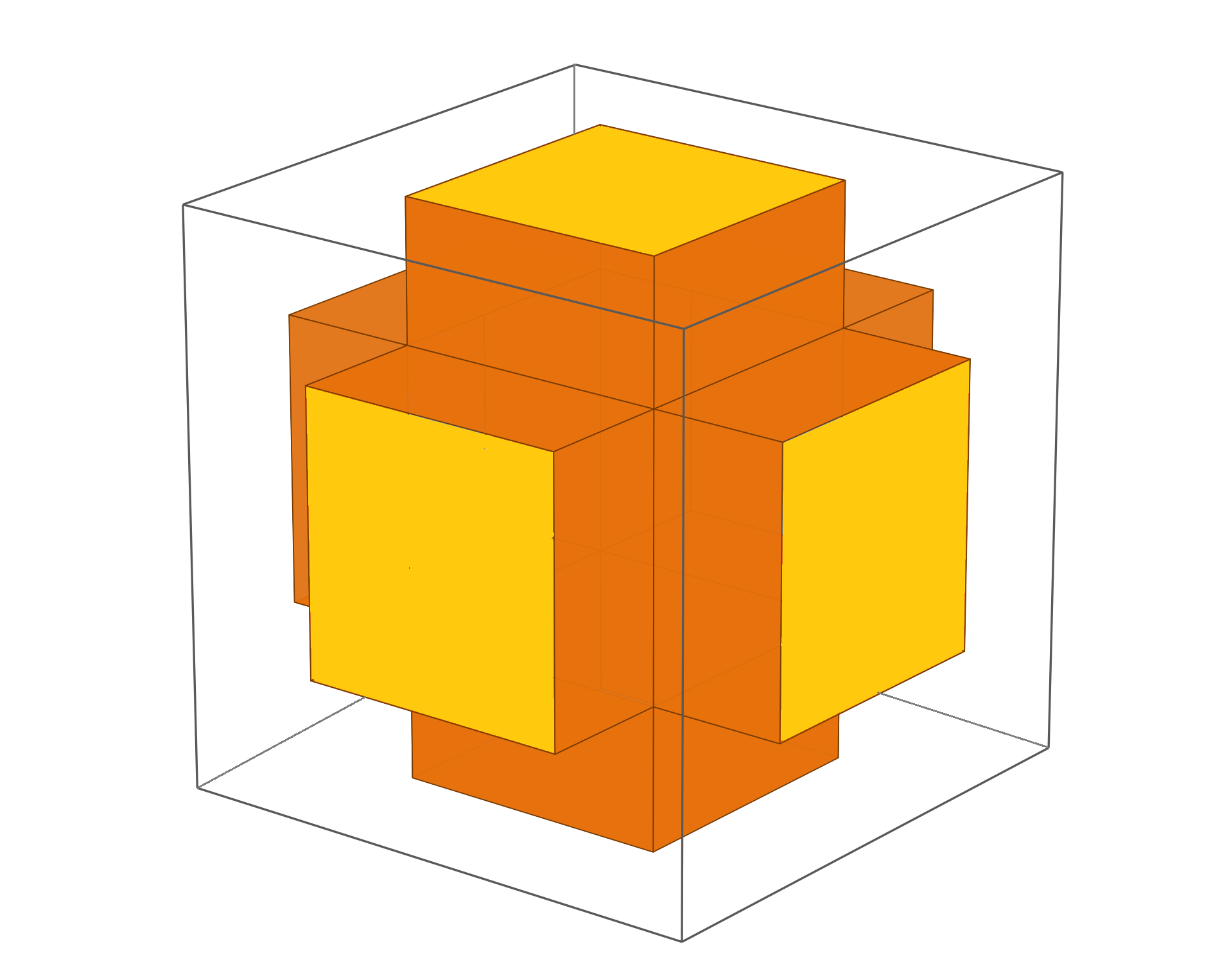}
    \caption{An image of $X^k$, $Z^k$, and $Y^{k-1}$ when $k=3$. Here, $Y^2$ is the orange surface, $X^3$ is the yellow part of the cube which lies `inside' $Y^2$, and the rest of the cube (`outside' of $Y^2$) is $Z^3$. Note when $k>3$, the boundary of $Y^{k-1}$ is connected, unlike the $k=3$ case pictured here.}
    \label{fig:cube-constructions}
\end{figure}

There is a canonical degree-one map $\pi: C^k \to C^k$ which projects $Z^k$ onto the $1$-skeleton $(C^k)^{(1)}$ of $C^k$, and stretches $X^k$ to take up all of $C^k$. This map is defined on $Z^k$ by applying the function 
$$
\mu(t) = \begin{cases}
    1 & \text{if }t \in [\frac{1}{2},1],\\
    2t & \text{if }t\in [-\frac{1}{2},\frac{1}{2}],\\
    -1 & \text{if }t\in [-1, -\frac{1}{2}],
\end{cases}
$$
to each coordinate, and it can be extended to all of $C^k$ in the natural way.

We will now state some necessary lemmas. Let $F^\ell$ denote one of the $\ell$-dimensional faces of $C^k$ for $\ell \in \{1,\dots,k\}$. In particular, $F^\ell \approx C^\ell$. By permuting coordinates and flipping signs if necessary, we can assume without loss of generality that $F^\ell \subseteq C^k$ is expressed in coordinates as
\begin{equation}\label{face-coord}
F^\ell = \{(1,\dots,1, x_{k-\ell + 1}, \dots, x_k) \in C^k : (x_{k-\ell+1}, \dots, x_k) \in C^\ell \}.
\end{equation}

\begin{lem}\label{face-lem}
For any $\ell$-dimensional face $F^\ell$ of $C^k$, we have
$$
Y^{k-1} \cap F^\ell \approx Y^{\ell - 1}.
$$
Similarly, 
$$
X^{k} \cap F^\ell \approx X^\ell.
$$
\end{lem}
\begin{proof}
    This follows directly from the set equations defining $X^k$ and $Y^{k-1}$ as well as (\ref{face-coord}).
\end{proof}

Any $\ell$-dimensional face $F^\ell$ of $C^k$ would be contained in some of the higher dimensional faces of $C^k$. Each of these higher dimensional faces is a cube, and $F^\ell$ would also be a face within these cubes.

\begin{lem}\label{face-count}
    Every $\ell$-dimensional face of $C^k$ is also a face of precisely $k-\ell$ of the $(\ell+1)$-dimensional faces of $C^k$.
\end{lem}

\begin{proof}
    Refer to equation (\ref{face-coord}). The face $F^\ell$ is in the boundary of any $(\ell +1)$-dimensional face of $C^k$ which has one less constant coordinate. Hence, picking an $(\ell+1)$-dimensional face of $C^k$ containing $F^\ell$ as a face amounts to picking one of the constant coordinates of $F^\ell$ to make variable, of which there are $k-\ell$ choices. 
\end{proof}

\subsection{Cubical Pseudomanifolds}

The cubical complexes $Y^{k-1}$ and $X^k$ will be the basic building blocks of the pseudomanifold $N$ mentioned at the beginning of this section; see Figure \ref{fig:N-pseudomanifold}. We recall the definition of cubical pseudomanifolds.

\begin{definition}
    An $m$-dimensional cubical pseudomanifold with boundary is a cubical $m$-complex~$P$ such that
\begin{itemize}
\item every cube of $P$ is a face of some $m$-cube of $P$;
\item every $(m-1)$-cube of $P$ is the face of at most two $m$-cubes of $P$;
\item given two $m$-cubes of $P$, there exists a sequence of $m$-cubes of $P$ with two consecutive $m$-cubes having an $(m-1)$-face in common that starts at one of them and ends at the other.
\end{itemize}
The boundary $\partial P$ of an $m$-dimensional cubical pseudomanifold $P$ is the cubical $(m-1)$-subcomplex of $P$ formed of the $(m-1)$-cubes of $P$ which are the faces of exactly one $m$-cube of $P$.
\end{definition}

Note that cubical pseudomanifolds have fundamental classes with any coefficients for which they are orientable.

It is clear that $X^k$ is a cubical $k$-pseudomanifold, and it was proven in \cite[Proposition 2.5]{nabutovsky2021sweepouts} that $Y^{k-1}$ is a cubical $(k-1)$-pseudomanifold with boundary lying inside the boundary $\partial C^k$. In fact, their proof also shows that the boundary $\partial Y^{k-1}$ precisely equals the intersection $Y^{k-1} \cap \partial C^k$. We can express $Y^{k-1}$ as a cubical chain with $G$ coefficients such that the boundary $\partial Y^{k-1}$ as a pseudomanifold is the same as the boundary $\partial Y^{k-1}$ as a chain. Moreover, for any face $F^{k-1}$ of $C^k$, recall from Lemma \ref{face-lem} that $Y^{k-1}\cap F^{k-1} \approx Y^{k-2}$, and it is straightforward to show that we can express each $Y^{k-1} \cap F^{k-1}$ as a cubical chain with coefficients in $G$ such that the following equation holds:

\begin{equation}\label{Y-bndry}
    \partial Y^{k-1} = \sum_{F^{k-1} \in  \partial C^{k}} Y^{k-1} \cap F^{k-1}.
\end{equation}

When we are gluing together the pieces $Y^{k-1}$ to construct $N$, equation (\ref{Y-bndry}) will be a key ingredient in ensuring we can use our decomposition to construct a filling of $S$. However, there is an important subtlety. In the ambient pseudomanifold $P$, the complex $Y^{k-1} \cap F^{k-1}$ is not well-defined as a cubical chain without specifying which $k$-cube $C_{(k)}$ in $P$ the $(k-1)$-cube $F^{k-1}$ is a face of. In turn, it would also depend on which $(k+1)$-cube of $P$ the $k$-cube $C_{(k)}$ is considered a face of, and so on. 

To overcome this issue, we define the notion of a $C_{(k)}$-sequence. Let $P^{(k)}$ denote the $k$-skeleton of a cubical pseudomanifold $P$.

\begin{definition}
    Let $P$ be a cubical $m$-pseudomanifold. For a fixed cube $C_{(k)}$ in $P^{(k)}$, we will call a \emph{$C_{(k)}$-sequence} a nested sequence 
$$
Q_{(k)} = \{C_{(k)} \subset C_{(k+1)} \subset  \dots \subset C_{(m)} \},
$$
where $C_{(i)} \in P^{(i)}$. Also, if we have a cube $C_{(k)}$ in $P^{(k)}$, a cube $C_{(k+1)}$ which contains $C_{(k)}$, and a $C_{(k+1)}$-sequence $Q_{(k+1)}$, we will denote by $Q_{(k+1)} +C_{(k)}$ the $C_{(k)}$-sequence which is made by adding $C_{(k)}$ into the nested sequence $Q_{(k+1)}$. 
\end{definition}

Since a $C_{(m)}$-sequence for a top dimensional cube only contains the cube $C_{(m)}$ itself, we will simply write $C_{(m)}$ in place of $Q_{(m)}$.

We are now in a position to construct the pseudomanifold $N^{m-1}$ which is decomposed nicely by its ambient cubical pseudomanifold $P^m$. It will be constructed by gluing together copies of $Y^{m-1}$ which sit in each top-dimensional cube of $P$, as well as gluing in copies of $X^{m-1}$ to give $N$ empty boundary. We will also need to prove some relations with chains in intermediate skeletons of $N$.

\begin{figure}[t]
    \centering
    \includegraphics[width=0.5\linewidth]{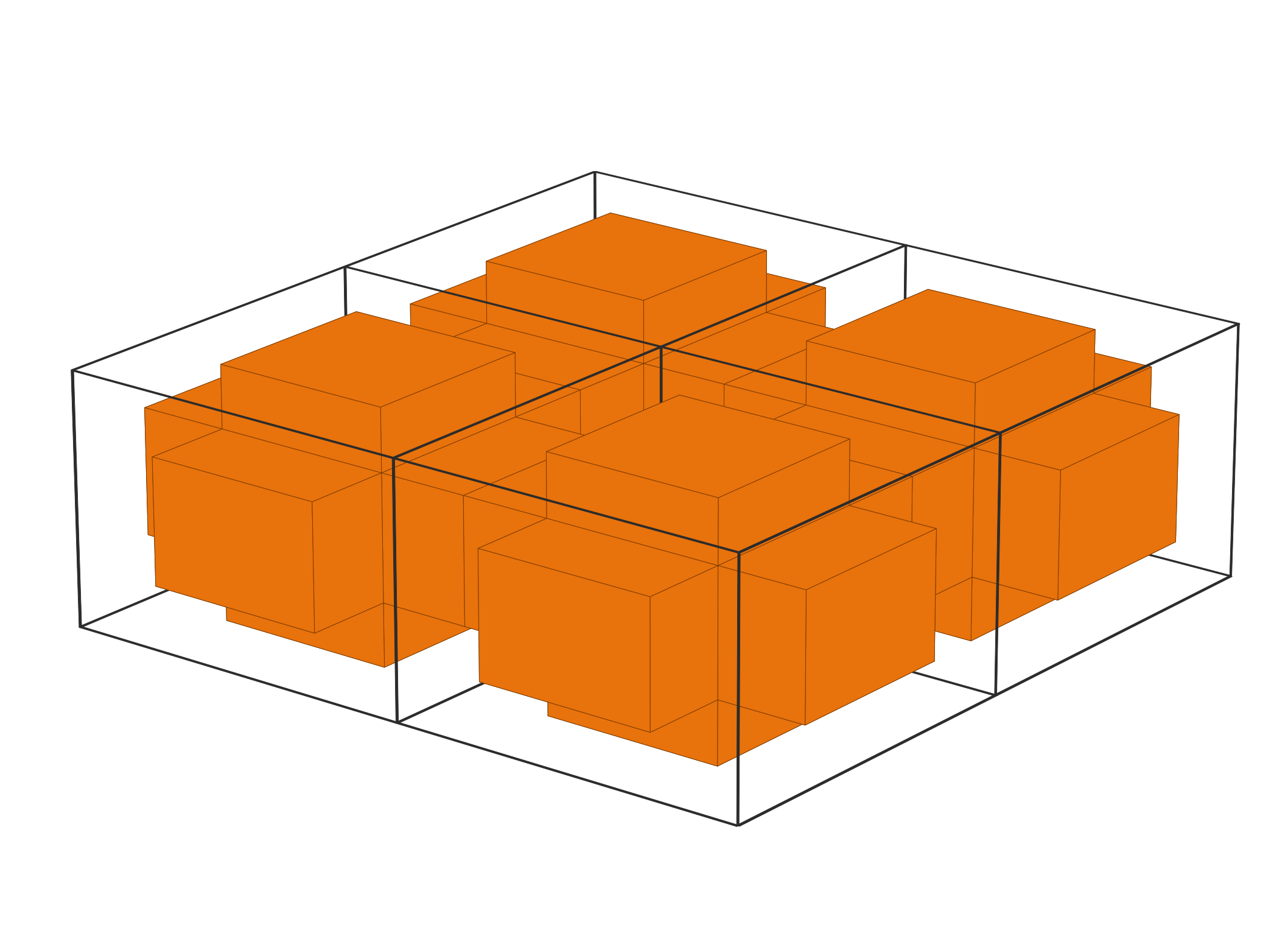}
    \caption{When $m=3$, the $2$-pseudomanifold $N$ is made up of copies of $Y^2$ glued together, along with copies of $X^2$ to cap the parts in the boundary of $P$. In this figure, the caps are the orange squares glued onto the boundary components of each $Y^3$ which lie in $\partial P$.}
    \label{fig:N-pseudomanifold}
\end{figure}
\begin{lem}\label{cube-decomp-lem}
    For any compact $m$-dimensional cubical pseudomanifold $P$ (with boundary) which is $G$-orientable, there exists a closed $(m-1)$-dimensional cubical pseudomanifold $N$ lying inside $P$, and for each $k$-cube $C_{(k)}$ in $P^{(k)} \setminus \partial P$ (with $k \geq 2$) and each $C_{(k)}$-sequence $Q_{(k)}$, there exists a $(k-1)$-chain $N_{Q_{(k)}}$ in $N$  such that:
    \begin{enumerate}[label=(\alph*)]
        \item 
        The chain $\sum_{C_{(m)} \in P^{(m)}} N_{C_{(m)}}$ represents a fundamental class in $H_{m-1}(N; G)$; 
        \item We have the equality
        $$
        \partial N_{Q_{(k)}} = \sum_{C_{(k-1)} \in \partial_{rel} C_{(k)}} N_{Q_{(k)} + C_{(k-1)}}
        $$
        as chains, where $\partial_{rel} C_{(k)}$ denotes the $(k-1)$-cubes in $\partial C_{(k)}$ which do not lie inside $\partial P$ (though they may intersect $\partial P$).
        \item If $Q_{(k)}$ and $Q'_{(k)}$ are two $C_{(k)}$-sequences which differ at only one entry, then
        $$
        N_{Q_{(k)}} + N_{Q'_{(k)}} = 0.
        $$
    \end{enumerate}
\end{lem}

\begin{proof}
    For each top dimensional cube $C_{(m)} \in P^{(m)}$, take a copy of $Y^{m-1}$ inside $C_{(m)}$. Glued together in the same way as all the cubes in $P^{(m)}$ are glued together to create $P$, we obtain a compact cubical $(m-1)$-dimensional pseudomanifold $N'$ with boundary lying in $\partial P$. For each $C_{(m-1)} \in P^{(m-1)} \cap \partial P$, the part of the boundary of $N'$ contained in $C_{(m-1)}$ is $Y^{m-1} \cap C^{m-1}$, which is isomorphic to $Y^{m-2}$ by Lemma \ref{face-lem}. Therefore, if we glue a copy of $X^{m-1}$ in each $C_{(m-1)} \in P^{(m-1)} \cap \partial P$, we fill the boundary of $N'$ and get a closed cubical $(m-1)$-dimensional pseudomanifold $N$. Since $P$ is $G$-orientable, so is $N$. 
    
    For each $C_{(m)} \in P^{(m)}$, let $N_{C_{(m)}} := N \cap C_{(m)}$. It follows directly from the construction of $N$ that we can give each  $N_{C_{(m)}}$ a chain structure (using the aforementioned chain structures on $Y^{m-1}$ and $X^{m-1}$) such that $\sum_{C_{(m)} \in P^{(m)}} N_{C_{(m)}}$ is a fundamental class in $H_{m-1}(N;G)$, hence (a) holds.

    Now, suppose $N_{Q_{(k)}}$ is defined for all $C_{(k)}$-sequences with $k>2$ and $C_{(k)} \not\in \partial P$. Let $Q_{(k-1)} = Q_{(k)} + C_{(k-1)}$ be a $C_{(k-1)}$-sequence with $C_{(k-1)} \not\in \partial P$. Using the chain structure on $Y^{k-1}$ which corresponds to that of $N_{Q_{(k)}}$, let $N_{Q_{(k-1)}}$ be the $(k-2)$-chain appearing in equation (\ref{Y-bndry}) which lies in the face $C_{(k-1)}$ (that is, $Y^{k-1} \cap F^{k-1}$ where $F$ is the face $C_{(k-1)}$ of $C_{(k)}$), plus corresponding copies of $X^{k-2}$ in any of the $(k-2)$-faces of $C_{(k-1)}$ which lie in $\partial P$. Doing this inductively, we get $(k-1)$-chains $N_{Q_{(k)}}$ for all $2 \leq k \leq m$. The equality
    $$
    \partial N_{Q_{(k)}} = \sum_{C_{(k-1)} \in \partial_{rel} C_{(k)}} N_{Q_{(k)} + C_{(k-1)}}
    $$
    follows from equation (\ref{Y-bndry}). The reason we ignore the cubes which lie in $\partial P$ is because the copies of $Y^{m-2}$ that lie in the $(m-1)$-dimensional cubes which make up the boundary $\partial P$ are filled with copies of $X^{m-1}$, and by Lemma \ref{face-lem} this means the copies of $Y^{k-1}$ which lie in the $k$-dimensional cubes in $\partial P$ are filled with copies of $X^k$. Hence, (b) is satisfied.

    Now let's prove (c). Suppose $Q_{(k)}$ and $Q'_{(k)}$ are two $C_{(k)}$-sequences which differ at only one entry. We want to prove
    $$
    N_{Q_{(k)}} + N_{Q'_{(k)}} = 0.
    $$
    We may assume without loss of generality that they differ at the $(k+1)$-dimensional cube containing $C_{(k)}$; since if the above equality is true, it is also true when we add lower-dimensional cubes to both $Q_{(k)}$ and $Q'_{(k)}$. Hence, there are only two cases to consider. The first is when $Q_{(k)}$ and $Q'_{(k)}$ differ at the top-dimensional cubes, i.e. say the sequences look like
    $$
    Q_{(k)} = \{ C_{(m-1)} \subset C_{(m)}\}
    $$
    and 
    $$
    Q'_{(k)} = \{C_{(m-1)} \subset C'_{(m)}\}.
    $$
    By our construction of $N$, we have the chain equality
    \begin{align*}
    0 &= \partial N\\
    &= \partial \left(\sum_{C_{(m)} \in P^{(m)}} N_{C_{(m)}}\right)\\
    &= \sum_{C_{(m)} \in P^{(m)}} \sum_{C_{(m-1)} \in \partial _*C_{(m)}} N_{C_{(m)} + C_{(m-1)}}.
    \end{align*}
    Since $P$ is a pseudomanifold, the cube $C_{(m-1)}$ appearing in $Q_{(k)}$ and $Q'_{(k)}$ appears only in the boundary of $C_{(m)}$ and $C'_{(m)}$. Hence,  $N_{Q_{(k)}}$ and $N_{Q'_{(k)}}$ are the only chains in the above equation whose image intersects the interior of $C_{(m-1)}$; since the equality holds as chains, this is only possible if
    $$
    N_{Q_{(k)}} + N_{Q'_{(k)}} = 0.
    $$
    The only other case to consider is when $Q_{(k)}$ and $Q'_{(k)}$ look like
    $$
    Q_{(k)} = \{C_{(k)} \subset C_{(k+1)} \subset C_{(k+2)} \subset \dots \subset C_{(m)} \}
    $$
    and
    $$
    Q'_{(k)} = \{C_{(k)} \subset C'_{(k+1)} \subset C_{(k+2)} \subset \dots \subset C_{(m)} \}.
    $$
    Let $Q_{(k+2)}$ be the sequence $\{C_{(k+2)} \subset \dots \subset C_{(m)}\}$. We have the chain equality
    \begin{align*}
    0 &= \partial^2N_{Q_{(k+2)}}\\
    &= \sum_{C_{(k+1)} \in \partial_{rel} C_{(k+2)}} \sum_{C_{(k)} \in \partial_{rel} C_{(k+1)}} N_{Q_{(k+2)} + C_{(k+1)} + C_{(k)}}.
    \end{align*}
    By Lemma \ref{face-count}, the cube $C_{(k)}$ appearing in $Q_{(k)}$ and $Q'_{(k)}$ is contained in exactly two of the $(k+1)$-faces of $C_{(k+2)}$; namely $C_{(k+1)}$ and $C'_{(k+1)}$. Hence, $N_{Q_{(k)}}$ and $N_{Q'_{(k)}}$ are the only chains in the above equation whose image intersects the interior of $C_{(k)}$; since the equality holds as chains, this is only possible if
    $$
    N_{Q_{(k)}} + N_{Q'_{(k)}} = 0.
    $$
    Therefore, we have proven (c).
\end{proof}

We will also need the following to prove the decomposition lemma.

\begin{lem}\label{R-complex}
    Inside the $m$-pseudomanifold $P$ from Lemma \ref{cube-decomp-lem}, there exists an $m$-pseudomanifold $R$ and a map $f:R \cup \partial P \to P^{(1)} \cup \partial P$ such that
    \begin{enumerate}
        \item\label{R-boundary} $\partial R = N - \partial P$;
        \item\label{R-map} For any top dimensional cube $C_{(m)}$ of $P$, $f(N_{C_{(m)}})$ consists of the $2^{m-1}m$ edges of $C_{(m)}$ plus the $(m-1)$-faces of $C_{(m)}$ lying in $\partial P$;
        \item For any cube $C_{(m-1)}$ in $\partial P$, $f|_{C_{(m-1)}} = \pi$ (where the map $\pi$ was defined in Section \ref{cube-subsec}).
    \end{enumerate}
\end{lem}
\begin{proof}
    Let $R$ be the cubical complex formed by gluing a copy of $Z^{m}$ in every top dimensional cube $C_{(m)}$ of $P$ (with $G$-orientation included from the orientation on $P$). Within each of these cubes, we have the map $\pi|_Z: Z \to (C_{(m)})^{(1)}$ defined in Section \ref{cube-subsec}, and these maps can be glued together to form a map $R \to P^{(1)}$. These maps agree with the maps $\pi: C_{(m-1)} \to C_{(m-1)}$ in each $(m-1)$-dimensional cube in $\partial P$, so we can extend our map $R \to P^{(1)}$ to a map $f: R \cup \partial P \to P^{(1)} \cup \partial P$. It is straightforward to verify that $R$ and $f$ satisfy the required properties.
\end{proof}

\subsection{Decomposition Lemma}
Our goal now is to construct a decomposition of our $d$-dimensional manifold $S$ into chains which have similar properties to the chains decomposing the cubical pseudomanifold $N$ in Lemma \ref{cube-decomp-lem}. This will be an instrumental tool in the proof of our main result. The way we will connect Lemma \ref{cube-decomp-lem} and the manifold $S$ is through the filling radius of $S$, the definition of which we now recall.

\begin{definition}
    The \emph{Kuratowski embedding} $i: S \hookrightarrow L^\infty(S)$ sends $x \in S$ to the function $d(x,\cdot): S \to \mathbb{R}$.  The \emph{filling radius} of $S$, denoted $\FillRad(S)$, is the infimal $r>0$ such that the cycle representing $i_*([S])$ in $L^\infty(S)$ is null-homologous (with $G$-coefficients) within its neighbourhood of radius $r$.
\end{definition}

We will often abuse notation and denote by $[S]$ the chain $i_*([S]) \subset L^\infty(S)$.

\begin{rem}\label{Kura}
    The Kuratowski embedding $i$ is an \emph{isometric imbedding}, meaning that for all $x,y \in S$, we have
    $$
    d_{S}(x,y) = d_{L^\infty}(x,y).
    $$
\end{rem}

The following theorem is an important result in metric geometry. The first proof was given by Gromov, and the best known constant was proven by Nabutovsky.

\begin{thm}[Gromov  \cite{gromov1983filling}, Nabutovsky \cite{nabutovsky2023linear}]\label{fillrad-thm}
    The inequality
    $$
    \FillRad(S) \leq \frac{d}{2} \Vol_d(S)^{\frac{1}{d}}
    $$
    holds for any closed $d$-dimensional manifold $S$.
\end{thm}

\begin{rem}\label{pseudo-fill}
    In the definition of filling radius, we used singular homology. This means, if $\nu >\FillRad(S)$, then there exists a singular chain in $L^\infty(S)$  whose boundary is $[S]$ and whose image sits inside the $\nu$-neighbourhood of $[S]$. Instead of expressing this chain as a formal sum of singular simplices, we can represent it by a map from a cubical pseudomanifold into $L^\infty(S)$ (see \cite[pp. 108-109]{hatcher2005algebraic}).
\end{rem}

We now prove the decomposition lemma, the main result of this section. It takes a similar form to Lemma \ref{cube-decomp-lem}, but with extra quantitative control.
\begin{lem}[Decomposition Lemma]\label{decomp-lem}
    Given a dimension $d$ and a constant $\delta > 0$, for any $\nu < \delta\, 2^{-(d+1)}(d+1)^{-1}$ the following holds. Suppose $S$ is any closed, $G$-orientable, $d$-dimensional manifold with $\FillRad(S) < \nu$. Then there exists a $G$-orientable, $(d+1)$-dimensional cubical pseudomanifold $P$, where for each $k$-cube $C_{(k)}$ in $P^{(k)} \setminus \partial P$ (with $k \geq 2$) and each $C_{(k)}$-sequence $Q_{(k)}$, there exists a $(k-1)$-chain $S_{Q_{(k)}}$ in $S$ such that: 
    \begin{enumerate}[label=(\alph*)]
        \item The chain $\sum_{C_{(d+1)} \in P^{(d+1)}} S_{C_{(d+1)}}$ represents a fundamental class in $H_d(S; G)$;
        \item We have the equality
        $$
        \partial S_{Q_{(k)}} = \sum_{C_{(k-1)} \in \partial_{rel} C_{(k)}} S_{Q_{(k)} + C_{(k-1)}}
        $$
        as chains, where $\partial_{rel} C_{(k)}$ denotes the $(k-1)$-cubes in $\partial C_{(k)}$ which do not lie inside $\partial P$ (though they may intersect $\partial P$).
        \item If $Q_{(k)}$ and $Q'_{(k)}$ are two $C_{(k)}$-sequences which differ at only one entry, then
        $$
        S_{Q_{(k)}} + S_{Q'_{(k)}} = 0.
        $$
        \item $S_{Q_{(k)}}$ is contained in a ball of radius $\delta$;
    \end{enumerate}
    
\end{lem}

\begin{proof}
Suppose we have chosen two parameters $\nu > \FillRad(S)$ and $\varepsilon > 0$ arbitrarily; we will later select these parameters carefully. By definition of filling radius and Remark \ref{pseudo-fill}, there is a $(d+1)$-dimensional cubical pseudomanifold $P$ and a map $e: P \to L^\infty(S)$ such that $e_*([\partial P]) = [S]$ and the image of $e$ lies in a $\nu$-neighbourhood of $S$. We can subdivide $P$ if necessary to assume that the image of each $(d+1)$-cube in $P^{(d+1)}$ has diameter less than $\varepsilon$ with the metric on $L^\infty(S)$. This implies the images of the $d$-cubes in $\partial P$, which make a cubical structure on $S$, each have diameter less than $\varepsilon$ (with the metric on $S$ by Remark \ref{Kura}).

By applying the lemmas to the $(d+1)$-pseudomanifold $P$, let $N$, $N_{Q_{(k)}}$, $R$, and $f: R \cup \partial P \to P^{(1)} \cup \partial P$ be as in Lemmas \ref{cube-decomp-lem} and \ref{R-complex}. We want to now define a map $g: P^{(1)} \cup \partial P \to S$. To map the $1$-skeleton $P^{(1)}$ into $S$, we start by mapping the $0$-skeleton $P^{(0)}$. For any $p \in P^{(0)}$, we define $g(p)$ to be any point in $S \subset L^\infty(S)$ which is closest to $e(p)$, where we recall $e: P \to L^\infty(S)$ is our map describing the filling of $S$ in $L^\infty(S)$ within radius $\nu$. Hence,  $d_{L^\infty}(e(p),g(p)) < \nu$. Now, extend this map to $P^{(1)}$ by mapping the edge connecting two adjacent vertices $p_1$ and $p_2$ in $P$ to a shortest path in $S$ connecting $g(p_1)$ and $g(p_2)$, chosen arbitrarily. In particular, after deforming $e$ if necessary so that it takes every edge of $\partial P$ to a minimizing segment of $S$, we can assume that $g$ and $e$ agree on the edges of $\partial P$. So we define $g$ on $\partial P$ as $e: \partial P \to S$, and overall we have a defined map $g: P^{(1)} \cup \partial P \to S$. Finally, denote $h := g \circ f: R \cup \partial P \to S$.

We show the restricted map $h|_{N}: N \to S$ is degree-one. Property (\ref{R-boundary}) of Lemma \ref{R-complex} says that $N$ and $\partial P$ are homologous within $R \cup \partial P$, with their difference bounding $R$. Hence,
$$
h_{*}([N]) = h_*([\partial P]) = (e_* \circ \overline{\pi}_*)([\partial P]),
$$
where $\overline{\pi}: \partial P \to \partial P$ is the map which applies $\pi: C_{(d)} \to C_{(d)}$ on every $d$-cube in $\partial P$. The second equality holds because $h|_{\partial P} = e \circ \overline{\pi}$. Since each $\pi$ is degree-one, so is $\overline{\pi}$, and therefore
$$
(e_* \circ \overline{\pi}_*)([\partial P]) = e_*([\partial P]) = [S].
$$

For any $k$, any $k$-cube $C_{(k)}$, and any $C_{(k)}$-sequence $Q_{(k)}$, define
$$
S_{Q_{(k)}} := h_*(N_{Q_{(k)}}).
$$ 
Properties (a)-(c) of the lemma follow directly from properties (a)-(c) of Lemma \ref{cube-decomp-lem} by simply applying $h_*$ (and for property (a) we use that $h$ is degree-one).

Lastly we need to prove property (d). The image of the chain $S_{Q_{(k)}}$ for any $C_{(k)}$-sequence $Q_{(k)}$ would be contained in $h(N_{C_{(k)}})$ (where $N_{C_{(k)}} := N \cap C_{(k)}$). Since this is contained in $h(N_{C_{(d+1)}})$ for some top dimensional cube $C_{(d+1)}$, it suffices to prove the diameter bound on any image $h(N_{C_{(d+1)}})$.
Pick an arbitrary cube $C_{(d+1)} \in P^{(d+1)}$. By Lemma \ref{R-complex}, the image $f(N_{C_{(d+1)}})$ consists of the $2^{d}(d+1)$ edges of $C_{(d+1)}$ plus the $d$-faces of $C_{(d+1)}$ lying in $\partial P$. Using Remark \ref{Kura} and the triangle inequality, we see that the image of each edge under $g$ has length bounded by
\begin{align*}
d_{S}(g(p_1),g(p_2)) &\leq d_{L^\infty}(g(p_1),e(p_1)) + d_{L^\infty}(e(p_1),e(p_2)) + d_{L^\infty}(e(p_2),g(p_2))\\
&< 2\nu + \varepsilon.
\end{align*}
Moreover, any $d$-face in the image $f(N_{C_{(d+1)}})$ is mapped to a $d$-cube in $S$ with diameter less than $\varepsilon$ by our previous refinement of $P$. Using the triangle inequality, we conclude that 
$$\diam(h(N_{C_{(d+1)}})) < 2^d(d+1)(2\nu + \varepsilon) + 2\varepsilon := \delta^*(\nu,\varepsilon).
$$
If we pick 
$$
\nu < \frac{\delta}{2^{d+1}(d+1)},
$$ then $\delta^*(\nu,\varepsilon) < \delta$ for all $\varepsilon$ sufficiently small. This proves property (d).
\end{proof}

\section{Filling Submanifolds}\label{proof-sec}

We are now in a position to fill $S$ within $M$. Recall that there exist constants
$(r_k,t_k)_{k=1}^{n-2}$ (with $r_k > 2t_{k-1}$) such that $M$ satisfies the filling conditions:
\begin{itemize}
    \item[(FC')] \emph{(Filling Conditions)} Any $k$-cycle which is contained in some ball of radius $r_k$ in $M$ is null-homologous with $G$-coefficients in the concentric ball of radius $t_k$.
\end{itemize}

The following proposition, together with Proposition \ref{ABG-prop} and Theorem \ref{fillrad-thm}, directly imply Theorem \ref{main-thm}.

\begin{prop}
    There exists a constant $\nu$, depending only on the dimension $n \geq 3$ and the filling conditions (FC'), such that the following holds.
    
    If $M$ is an $n$-manifold satisfying (FC'), and $S$ is a closed embedded $(n-2)$-submanifold (orientable if $M$ is) with $\FillRad(S) < \nu$, then $S$ is null-homologous within a neighbourhood of radius $2t_{n-2}$ in $M$.
\end{prop}
\begin{proof}
    Let $d = n-2$. Pick $\delta>0$ such that
    \begin{equation}\label{delta}
    \delta < \min_{i \in \{1,\dots,d\}}\{r_i - 2t_{i-1}\}
    \end{equation}
    (when $i=1$, $t_{i-1}$ is understood to equal zero). With this $\delta$, pick the constant $\nu >0$ from Lemma \ref{decomp-lem}. Now suppose $S$ is a closed $d$-dimensional submanifold of $M$ with $\FillRad(S) < \nu$.

    Let $P$ be as in Lemma \ref{decomp-lem}. It will be helpful to define some new notation. Let $\mathcal{I}^k$ denote the $k$-cubes in $P^{(k)}$ which are not a subset of $\partial P$. And let $\mathcal{I}_{C_{(k)}}^m$, for $m < k$, denote the $m$-cubes in $\mathcal{I}^m$ which are contained in the $m$-skeleton of $C_{(k)}$. We ignore the cubes lying in $\partial P$ because we do not need to fill any cycles there during the proof, as they were already filled by our construction of $N$ (by gluing in copies of $X^d$). Consequently, all the $C_{(k)}$-sequences we consider will only contain cubes not contained in the boundary, i.e. will contain cubes in $\mathcal{I}^m$.
    
    We will prove by induction on $k \in \{2,\dots, d \}$ that for every $k$-cube $C_{(k)} \in \mathcal{I}^k$ and every $C_{(k)}$-sequence $Q_{(k)}$, there exists a $(k-1)$-cycle $Z_{Q_{(k)}}$ in $M$ such that the following conditions are satisfied:
    \begin{enumerate}[label=(\roman*)]
        \item Each $Z_{Q_{(k)}}$ is contained in a ball of radius $r_{k-1}$;
        \item For any $C_{(k+1)}$-sequence $Q_{(k+1)}$, 
        $$
        \partial S_{Q_{(k+1)}} = \sum_{C_{(k)} \in \mathcal{I}_{C_{(k+1)}}^{k}} Z_{Q_{(k+1)} + C_{(k)}}; 
        $$
        \item Suppose $Q_{(k)}$ and $Q'_{(k)}$ are two $C_{(k)}$-sequences which differ at only one entry. Then 
        $$
        Z_{Q_{(k)}} + Z_{Q'_{(k)}} = 0.
        $$
    \end{enumerate}

    We start with the $k=2$ case. For any $C_{(2)} \in \mathcal{I}^2$ and any $C_{(2)}$-sequence $Q_{(2)}$, simply let $Z_{Q_{(2)}} := S_{Q_{(2)}}$. Then $Z_{Q_{(2)}}$ is a $1$-cycle (see the equation defining $Y^{k-1}$ in Section \ref{fillrad-sec}), and condition (i) follows from property (d) of Lemma \ref{decomp-lem} along with inequality (\ref{delta}). Conditions (ii) and (iii) are the same as properties (b) and (c) respectively from Lemma \ref{decomp-lem}. 

    Now suppose for our induction hypothesis that for some $k \in \{2,\dots,d-1 \}$ and all $C_{(k)}$-sequences $Q_{(k)}$ there exist $(k-1)$-cycles $Z_{Q_{(k)}}$ satisfying (i)-(iii). We want to construct, for all $C_{(k+1)} \in \mathcal{I}^{k+1}$ and any $C_{(k+1)}$-sequence $Q_{(k+1)}$, the $k$-cycles $Z_{Q_{(k+1)}}$ satisfying (i)-(iii).

    Let $Q_{(k+1)}$ be a $C_{(k+1)}$-sequence. For any $C_{(k)} \in \mathcal{I}^k_{C_{(k+1)}}$, the $(k-1)$-cycle $Z_{Q_{(k+1)} + C_{(k)}}$ is contained in a ball of radius $r_{k-1}$ by (i) of the induction hypothesis. By (FC'), this means there exists a $k$-chain $F_{Q_{(k+1)} + C_{(k)}}$ contained in a ball of radius $t_{k-1}$ such that 
    $$
    \partial F_{Q_{(k+1)} + C_{(k)}} = Z_{Q_{(k+1)} + C_{(k)}},
    $$
    and by property (iii) of the induction hypothesis, we can do this in such a way that, for any $C_{(k)}$-sequences $Q_{(k)}$ and $Q'_{(k)}$ which differ at one entry, we have
    \begin{equation}\label{fills-cancel}
    F_{Q_{(k)}} + F_{Q'_{(k)}} = 0.
    \end{equation} 
    
    Define the $k$-chain
    $$
    Z_{Q_{(k+1)}} := S_{Q_{(k+1)}} - \sum_{C_{(k)} \in \mathcal{I}_{C_{(k+1)}}^{k}} F_{Q_{(k+1)}+C_{(k)}}.
    $$
    Applying property (ii) of the induction hypothesis, we see $\partial Z_{Q_{(k+1)}} = 0$ so $Z_{Q_{(k+1)}}$ is a cycle. The chain $S_{Q_{(k+1)}}$ is contained in a ball of radius $\delta$ by property (d) of Lemma \ref{decomp-lem}, and each $F_{Q_{(k+1)}+C_{(k)}}$ is contained in a ball of radius $t_{k-1}$, thus the cycle $Z_{Q_{(k+1)}}$ is contained in a ball of radius $\delta + 2t_{k-1}$. By inequality (\ref{delta}), this means $Z_{Q_{(k+1)}}$ is contained in a ball of radius $r_{k}$. This proves (i) is satisfied.

    To prove (ii), pick some arbitrary $C_{(k+2)} \in \mathcal{I}^{k+2}$ and a $C_{(k+2)}$-sequence $Q_{(k+2)}$. We have
    \begin{align*}
        &\sum_{C_{(k+1)} \in \mathcal{I}_{C_{(k+2)}}^{k+1}} Z_{Q_{(k+2)} + C_{(k+1)}} \\&= \sum_{C_{(k+1)} \in \mathcal{I}_{C_{(k+2)}}^{k+1}} \left(S_{Q_{(k+2)} + C_{(k+1)}} - \sum_{C_{(k)} \in \mathcal{I}_{C_{(k+1)}}^{k}} F_{Q_{(k+2)} + C_{(k+1)} + C_{(k)}} \right) \\
        &= \left(\sum_{C_{(k+1)} \in \mathcal{I}_{C_{(k+2)}}^{k+1}} S_{Q_{(k+2)} +C_{(k+1)}} \right)- \left(\sum_{C_{(k+1)} \in \mathcal{I}_{C_{(k+2)}}^{k+1}}\sum_{C_{(k)} \in \mathcal{I}_{C_{(k+1)}}^{k}} F_{Q_{(k+2)} + C_{(k+1)} + C_{(k)}}\right).
    \end{align*}
    The first term above equals $\partial S_{Q_{(k+2)}}$ by condition (b) of Lemma \ref{decomp-lem}. The second term equals zero. To see this, note that a $k$-cube $C_{(k)}$ which is in $\mathcal{I}_{C_{(k+2)}}^k$ appears in precisely two of the $(k+1)$-faces of $C_{(k+2)}$ by Lemma \ref{face-count}. Hence,
    \begin{align*}
        & \sum_{C_{(k+1)} \in \mathcal{I}_{C_{(k+2)}}^{k+1}}\sum_{C_{(k)} \in \mathcal{I}_{C_{(k+1)}}^{k}} F_{Q_{(k+2)} + C_{(k+1)} + C_{(k)}}\\ &= \sum_{C_{(k)} \in \mathcal{I}^{k}_{C_{(k+2)}}} \left( F_{Q_{(k+2)} + C_{(k+1)} + C_{(k)}} + F_{Q_{(k+2)} + C'_{(k+1)} + C_{(k)}}\right)\\
        &=0,
    \end{align*}
    with the last equality following from equation (\ref{fills-cancel}). 

    Finally, we prove property (iii). Suppose $Q_{(k+1)}$ and $Q'_{(k+1)}$ are two $C_{(k+1)}$-sequences which differ in exactly one spot. We have 
    \begin{align*}
        &Z_{Q_{(k+1)}} + Z_{Q'_{(k+1)}}\\ &= \left(S_{Q_{(k+1)}} + S_{Q'_{(k+1)}} \right) -\sum_{C_{(k)} \in \mathcal{I}^k_{C_{(k+1)}}} \left(F_{Q_{(k+1)}+ C_{(k)}} + F_{Q'_{(k+1)}+ C_{(k)}} \right)\\
        &= 0,
    \end{align*}
    with the last equality following from property (c) of Lemma \ref{decomp-lem} and equation (\ref{fills-cancel}). This finishes the induction proof.

    Therefore, we have a collection of $(d-1)$-cycles $Z_{Q_{(d)}}$, for all $C_{(d)}$-sequences $Q_{(d)}$, such that properties (i)-(iii) hold. We mostly reuse the same proof of the induction step one last time, with slight modifications, to complete the proof of the proposition. We can assume we have $n>3$ and $d > 1$; the $n=3,d=1$ case simply follows the strategy from the inductive proof.
    
    Each $Z_{Q_{(d)}}$ is contained in a ball of radius $r_{d-1}$ by (i). By (FC'), this means there exists a $d$-chain $F_{Q_{(d)}}$ contained in a ball of radius $t_{d-1}$ such that $\partial F_{Q_{(d)}} = Z_{Q_{(d)}}$. For each $C_{(d+1)} \in \mathcal{I}^{d+1}= P^{(d+1)}$, define the $d$-chain
    $$
    Z_{C_{(d+1)}} := S_{C_{(d+1)}} - \sum_{C_{(d)} \in \mathcal{I}_{C_{(d+1)}}^{d}} F_{C_{(d+1)} + C_{(d)}}.
    $$
    By property (ii), $Z_{C_{(d+1)}}$ is a cycle. The chain $S_{C_{(d+1)}}$ is contained in a ball of radius $\delta$ by property (d) of Lemma \ref{decomp-lem}, and each $F_{C_{(d+1)} + C_{(d)}}$ is contained in a ball of radius $t_{d-1}$ by property (i), thus the cycle $Z_{C_{(d+1)}}$ is contained in a ball of radius $\delta + 2t_{d-1}$. By inequality (\ref{delta}), this means $Z_{C_{(d+1)}}$ is contained in a ball of radius $r_{d}$. Hence, there exists a $(d+1)$-chain $F_{C_{(d+1)}}$ within a ball of radius $t_{d}$ such that
    $$
    \partial F_{C_{(d+1)}} = Z_{C_{(d+1)}}.
    $$
    We have
    \begin{align*}
        &\sum_{C_{(d+1)}\in P^{(d+1)}} Z_{C_{(d+1)}} \\&= \sum_{C_{(d+1)}\in P^{(d+1)}} \left(S_{C_{(d+1)}} -\sum_{C_{(d)} \in \mathcal{I}_{C_{(d+1)}}^{d}} F_{C_{(d+1)} + C_{(d)}}
        \right)\\
        &= \left(\sum_{C_{(d+1)}\in P^{(d+1)}} S_{C_{(d+1)}} \right) - \left(\sum_{C_{(d+1)}\in P^{(d+1)}} \sum_{C_{(d)} \in \mathcal{I}_{C_{(d+1)}}^{d}} F_{C_{(d+1)} + C_{(d)}}\right)
    \end{align*}
    The first term equals $[S]$ in homology by property (a) of Lemma \ref{decomp-lem}. The second term equals zero: since the $(d+1)$-dimensional cubical complex $P$ is a pseudomanifold, any $C_{(d)} \in \mathcal{I}_{C_{(d+1)}}^d$ is going to appear as a face of exactly two $(d+1)$-cubes in $P^{(d+1)}$, and this term will evaluate to zero using property (iii) in the same way as in the induction step. Therefore,
    $$
    \sum_{C_{(d+1)}\in P^{(d+1)}}Z_{C_{(d+1)}} = [S].
    $$
    Hence, the $(d+1)$-chain
    $$
    F := \sum_{C_{(d+1)}\in P^{(d+1)}} F_{C_{(d+1)}}
    $$
    has $[S]$ as its boundary, and the fact that each $F_{C_{(d+1)}}$ is in a ball of radius $t_d$ implies the chain $F$ lies within a neighbourhood of radius $2t_d$ in $S$. 
\end{proof}

\section{Macroscopic Dimension and Topology}\label{top-sec}
In this section, we prove Theorem \ref{topcor}. The following result seems to be known to experts:

\begin{thm}\label{virtfree-thm}
    Suppose $M$ is a closed Riemannian manifold with $\dim_{mc}(\widetilde{M})\leq 1$. Then $\pi_1(M)$ is virtually free.
\end{thm}
This is a stronger version of \cite[Corollary 14]{chodosh2023classifying}. To prove it, we will use some geometric group theory. The \emph{number of ends of a group $G$} is the number of ends of $K$ for any covering map $K \to K_0$ into a compact space $K_0$ which has $G$ as its group of deck transformations. If $G$ is finitely generated, this is equal to the number of ends of any Cayley graph of $G$. Recall that the number of ends of a topological space $K$ is the minimal number of connected components of $K\setminus C$ over all compact sets $C \subseteq K$.

We first need a preliminary lemma.
\begin{lem}
    If $\widetilde{M}$ is the universal cover of a Riemannian manifold which has $\UW_1(\widetilde{M}) < w$, then there exists a piecewise-linear map $f: \widetilde{M} \to P$ into a graph such that $\diam f^{-1}(z) < w$ for all fibers $f^{-1}(z)$.
\end{lem}
\begin{proof}
    We follow the proof of Lemma 3.9 from \cite{balitskiy2022waist}. Let $g:\widetilde{M} \to P$ be a continuous, proper map with diameter of fibers less than $w$. Subdivide $P$ finely so that the preimage of the open star of each vertex has diameter less than $w$ (using local finiteness of $P$, which we have without loss of generality by the symmetry of $\widetilde{M}$). Use the locally finite simplicial approximation theorem (see \cite[pg. 128]{spanier2012algebraic}) to construct a piecewise-linear map $f: \widetilde{M} \to P$ such that for each $x \in \widetilde{M}$, $f(x)$ belongs to the minimal closed cell in $P$ containing $g(x)$. This implies each fiber is contained in the open star of some vertex, hence has diameter less than $w$.
\end{proof}

\begin{prop}\label{ends-prop}
    Suppose $M$ is a closed Riemannian manifold with $\dim_{mc}(\widetilde{M})\leq 1$. Then any finitely generated subgroup of $\pi_1(M)$ cannot have one end.
\end{prop}

This result is stronger than \cite[Corollary 14]{chodosh2023classifying}, but uses a similar proof.

\begin{proof}
    Suppose for the sake of contradiction that there was a finitely generated subgroup $G$ of $\pi_1(M)$ with one end. Let $M_1 \to M$ be a covering of $M$ with $\pi_1(M_1) \approx G$. Since $G$ is finitely generated, we can find a compact set $K_1$ in $M_1$ with the inclusion map $i_1: K_1 \to M_1$ inducing a surjective map on fundamental groups. That is, $(i_1)_*: \pi_1(K_1) \to \pi_1(M_1) = G$ is surjective, and therefore $G \approx \pi_1(K_1) / \ker(i_1)_*$. Find a covering $p_2: K_2 \to K_1$ with $\pi_1(K_2) \approx \ker(i_1)_*$. This means that $(p_2)_*(\pi_1(K_2))$ is normal in $\pi_1(K_1)$, so the group of deck transformations of $K_2 \to K_1$ is $\pi_1(K_1) / \pi_1(K_2) \approx   \pi_1(K_1) / \ker(i_1)_* \approx G$. Since $K_1$ is compact, and $G$ has one end, this all implies $K_2$ has one end. Since $(i_1)_*\circ(p_2)_*$ is trivial, we can lift this map to a map $i_2: K_2 \to \widetilde{M}$, where $\widetilde{M}$ is the universal cover of $M$. Therefore, we have the following diagram:
    $$
    \begin{array}{ccc}
        K_2 & \rightarrow &\widetilde{M}\\
        \downarrow{} & & \downarrow{}\\
        K_1 & \rightarrow & M_1\\
        & & \downarrow{}\\
        & & M
    \end{array}
    $$
    
We will need the following lemmas, which are identical to Lemmas 15, 16, and 17 in \cite{chodosh2023classifying} (and have identical proofs).    \begin{lem}\label{propemb}
        The map $i_2$ is a proper embedding.
    \end{lem}

    \begin{lem}\label{distbound}
        For each $r>0$, there exists $R(r)>0$ such that, for any $a,b \in K_2$ with $d_{\widetilde{M}}(a,b) \leq r$, we have $d_{K_2}(a,b) \leq R(r)$.
    \end{lem}

    \begin{lem}
        There exists a line $\gamma_2$ in $K_2$.
    \end{lem}

    Parameterize the curve $\gamma_2$ so that $d_{K_2}(\gamma_2(s),\gamma_2(t)) = |s-t|$ (note that $d_{\widetilde{M}}(\gamma_2(s),\gamma_2(t))$ may be less than $|s-t|$). Let $\tilde{\gamma} = i_2 \circ \gamma_2$. Note that $\gamma_2$ is automatically proper in $K_2$, and thus Lemma \ref{propemb} tells us that $\tilde{\gamma}$ is proper in $\widetilde{M}$.
    
    Let $f: \widetilde{M} \to P$ be a piecewise-linear map from the universal cover $\widetilde{M}$ to a graph $P$ with width less than $w$ for some fixed $w$. Let $z := f(\tilde{\gamma}(0)) \in P$, and fix a path metric $d_P$ on $P$ so that we can consider closed balls $B_r(z)$ in $P$ centered at $z$. Define the parameters
    $$
    \begin{array}{l}
        t_-(r) := \max\{t: f(\tilde{\gamma}(-\infty,t)) \cap B_r(z) = \emptyset\},\\
        t_+(r) := \min \{t: f(\tilde{\gamma}(t,\infty)) \cap B_r(z) = \emptyset \}.
    \end{array}
    $$
    Note that $t_{\pm}(r) \to \pm \infty$ as $r \to \infty$, since the balls $B_r(z)$ exhaust the component of $P$ which contains the image of $\tilde{\gamma}(\mathbb{R})$ under $f$.

    Let $C_r := f^{-1}(B_r(z))$. Since $\partial B_r(z)$ is a finite collection of points in $P$, $\partial C_r$ is a disjoint union of fibers of $f$, and by piecewise-linear Sard's lemma, take generic $r$ so that these fibers are $n-1$-dimensional piecewise-linear submanifolds. We claim that both $\widetilde{\gamma}(t_{\pm}(r))$ must lie in the same fiber in $\partial C_r$. Suppose this were not the case. Take the path $\widetilde{\gamma}([t_-(r),t_+(r)])$, and using the fact $K_2$ has one end, connect the endpoints $\tilde{\gamma}(t_{\pm}(r))$ of this path with another path in $K_2 \setminus C_r$ such that we create a closed loop $\tau$ which transversely intersects the two fibers containing $\widetilde{\gamma}(t_{\pm}(r))$. Then $\tau$ would have algebraic intersection number one with each of these fibers, but this contradicts that the loop $\tau$ is contractible in the universal cover $\widetilde{M}$. This proves our claim that $\widetilde{\gamma}(t_\pm(r))$ both lie in the same fiber of $f$. Hence,
    $$
    d_{\widetilde{M}}(\tilde{\gamma}(t_-(r)),\tilde{\gamma}(t_+(r))) < w.
    $$
    By Lemma \ref{distbound}, this implies
    $$
    d_{K_2}(\tilde{\gamma}(t_-(r)),\tilde{\gamma}(t_+(r))) \leq R(w).
    $$
    This holds for generic, sufficiently large $r$. On the other hand, we have
    $$
    d_{K_2}(\tilde{\gamma}(t_-(r)),\tilde{\gamma}(t_+(r))) = \left|t_-(r) - t_+(r) \right| \to \infty
    $$
    as $r \to \infty$. This gives our contradiction.
\end{proof}

This puts us in a position to prove Theorem \ref{virtfree-thm} and Theorem \ref{topcor}.
\begin{proof}[Proof of Theorem \ref{virtfree-thm}] Since the fundamental group of any closed manifold is finitely presented, Dunwoody's result in geometric group theory \cite{dunwoody1985accessibility} says that $\pi_1(M)$ is the fundamental group of a graph of groups, where each edge group is finite and each vertex group has at most one end (see \cite[Section 3]{scott1979topological} for the definition of graphs of groups and their fundamental groups). By Proposition \ref{ends-prop}, the vertex groups cannot have one end, and therefore they must be finite. Then \cite[Theorem 7.3]{scott1979topological} implies that $\pi_1(M)$ is virtually free.
    
\end{proof}

\begin{proof}[Proof of Theorem \ref{topcor}]
    By Theorem \ref{main-thm} and Theorem \ref{virtfree-thm}, $\pi_1(M)$ is virtually free. Let $G \subseteq \pi_1(M)$ be a finite index subgroup which is a free group. Consider the finite covering $\widehat M\xrightarrow{\hat p} M$ such that $\hat{p}_*(\pi_1(\widehat{M})) = G$. Then $\pi_1(\widehat M)$ is a finitely generated free group. Since $\pi_2(\widehat M)=\cdots=\pi_{n-2}(\widehat M)=0$, \cite[Section 2 and Section 3]{gadgil2009topology} implies that $\widehat M$ is homotopy equivalent to $S^n$ or connected sums of $S^{n-1}\times S^1$.
\end{proof}

\printbibliography

\end{document}